\documentclass[reqno]{amsart}
\usepackage{amsmath, amsthm, amssymb, mathabx}
\begin{document}
\newtheorem{theo}{Theorem}[section]
\newtheorem*{thm*}{Theorem A}
\newtheorem{defin}[theo]{Definition}
\newtheorem{rem}[theo]{Remark}
\newtheorem{lem}[theo]{Lemma}
\newtheorem{cor}[theo]{Corollary}
\newtheorem{prop}[theo]{Proposition}
\newtheorem{exa}[theo]{Example}
\newtheorem{exas}[theo]{Examples}
\numberwithin{equation}{section}
%
%
\thanks{}
\title[Multiple sign-changing solutions]{Multiple sign-changing solutions to a class of Kirchhoff type problems}

\author[]{Cyril Joel Batkam}
\address{Cyril Joel Batkam
 \newline
HEC Montreal,
\newline
3000 Chemin de la Côte-Sainte-Catherine
\newline
 Montréal, QC, H3T 2B1, Canada}
\email{cyril-joel.batkam@hec.ca}
\maketitle
\begin{abstract}
This paper is concerned with the existence of sign-changing solutions to nonlocal Kirchhoff type problems of the form
\begin{equation}\label{s}\tag{S}
 -\Big(a+b\int_\Omega|\nabla u|^2dx\Big)\Delta u=f(x,u)\, \text{ in }\Omega,\quad\quad u=0 \text{ on }\partial\Omega,
\end{equation}
where $\Omega$ is a bounded domain in $\mathbb{R}^N$ ($N=1,2,3$) with smooth boundary, $a>0$, $b>0$, and  $f:\overline{\Omega}\times\mathbb{R}\to\mathbb{R}$ is a continuous function. We first establish a new sign-changing version of the symmetric mountain pass theorem and then apply it to prove the existence of a sequence of sign-changing solutions to \eqref{s} with higher and higher energy.
\end{abstract}
%
\section{Introduction}
In this paper, we investigate the multiplicity of sign-changing solutions to nonlocal Kirchhoff type problems of the form
\begin{equation}\label{s}\tag{S}
 -\Big(a+b\int_\Omega|\nabla u|^2dx\Big)\Delta u=f(x,u)\, \text{ in }\Omega,\quad\quad u=0 \text{ on }\partial\Omega,
\end{equation}
where $\Omega$ is a bounded domain in $\mathbb{R}^N$ ($N=1,2,3$) with smooth boundary, $a>0$, $b\geq0$, and  $f:\overline{\Omega}\times\mathbb{R}\to\mathbb{R}$ is a nonlinear function.

This problem is related to the stationary analogue of the hyperbolic equation
\begin{equation*}
  u_{tt}-\Big(a+b\int_\Omega|\nabla u|^2dx\Big)\Delta u=f(x,u),
\end{equation*}
which is a general version of the equation
\begin{equation}\label{kir}
    \rho\frac{\partial^2u}{\partial t^2}-\Big(\frac{\rho_0}{h}+\frac{E}{2L}\int_0^L\big|\frac{\partial u}{\partial x}\big|^2dx\Big)\frac{\partial^2u}{\partial x^2}=0
\end{equation}
proposed by Kirchhoff \cite{Kir} as an extension of the classical D'Alembert's wave equation for free vibrations of elastic strings. This model takes into account the changing in length of the string produced by transverse vibrations. In \eqref{kir}, $L$ is the length of the string, $h$ is the area of the cross-section, $E$ is the Young's modulus of the material, $\rho$ is the mass density, and $\rho_0$ is the initial tension.

 When $b>0$, problem \eqref{s} is said to be nonlocal. In that case, the first equation in \eqref{s} is no longer a pointwise equality. This causes some mathematical difficulties which make the study of such problems particularly interesting. Some early classical studies of Kirchhoff type problems can be found in \cite{Bern,Poho}. However, problem \eqref{s} received much attention only after the paper of Lions \cite{Lions}, where an abstract framework to attack it was introduced. Some existence and multiplicity results can be found in \cite{Per-Zhang,He-Zou,B13,LiuHe12} without any information on the sign of the solutions. Recently, Alves et al. \cite{Alves05}, Ma and Rivera \cite{Ma-Riv03}, and Cheng and Wu \cite{Cheng-Wu} obtained one positive solution. In \cite{HeZou09}, He and Zou obtained infinitely many positive solutions. The existence of sign-changing solutions to \eqref{s} was considered by Figuereido and Nascimento \cite{Figue}, Perera and Zhang \cite{Per-Zhang06}, Mao and Zhang \cite{Mao-Zhang09}, and Mao and Luan \cite{Mao-Luan11}. But only one sign-changing solution was found in these papers. In case $f$ is a pure power type nonlinearity, Alves et al. \cite{Alves05} related the number of solutions of \eqref{s} to that of a local problem by using a scaling argument. As a consequence, one can obtain in that particular case infinitely many sign-changing solutions (see \cite{WHL}). However, the scaling approach does not provide high energy solutions even in the simple case of power type nonlinearity.

  In this paper, we develop a variational approach to study high-energy sing-changing solutions to some classes of nonlocal problems.
 
 Our result on \eqref{s} relies on the following standard conditions on the nonlinear term $f$:
\begin{enumerate}
  \item[$(f_1)$] $f:\overline{\Omega}\times\mathbb{R}\to\mathbb{R}$ is continuous and there exists a constant $c>0$ such that
\begin{equation*}
    |f(x,u)|\leq c\big(1+|u|^{p-1}\big),
\end{equation*}
 where $p>4$ for $N=1,2$ and $4<p<6$ for $N=3$.
  \item[$(f_2)$] $f(x,u)=\circ(|u|)$, uniformly in $x\in \overline{\Omega}$, as $u\to0$.
  \item[$(f_3)$] there exists $\mu>4$ such that $0<\mu F(x,u)\leq uf(x,u)$ for all $u\neq0$ and for a.e $x\in\overline{\Omega}$, where $F(x,u)=\int_0^u f(x,s)ds$.
  \item [$(f_4)$] $f(x,-u)=-f(x,u)$ for all $(x,u)\in\overline{\Omega}\times\mathbb{R}$.
\end{enumerate}
Our result reads as follows:
 \begin{theo}\label{mainresult}
 Let $a>0$ and $b\geq0$. Assume that $f$ satisfies the conditions $(f_{1,2,3,4})$.
Then \eqref{s} possesses a sequence $(u_k)$ of sign-changing solutions such that
\begin{equation*}
\frac{a}{2}\int_\Omega|\nabla u_k|^2dx+\frac{b}{4}\Big(\int_\Omega|\nabla u_k|^2dx\Big)^2-\int_\Omega F(x,u_k)dx\to+\infty,\quad \text{as }k\to\infty.
\end{equation*}
\end{theo}
If $b=0$, we obtain the following consequence of the above result.
\begin{cor}\label{zero}
Under assumptions $(f_{1,2,3,4})$, the semilinear problem
\begin{equation}\label{classic}
 -\Delta u=f(x,u)\, \text{ in }\Omega,\quad\quad u=0 \text{ on }\partial\Omega,
\end{equation}
 possesses a sequence $(u_k)$ of sign-changing solutions such that
\begin{equation*}
\frac{1}{2}\int_\Omega|\nabla u_k|^2dx-\int_\Omega F(x,u_k)dx\to+\infty,\quad \text{as }k\to\infty.
\end{equation*}
\end{cor}
Note that Corollary \ref{zero} was obtained by Qian and Li in \cite{QianLi04} by means of the method of invariant sets of descending flow. Earlier proofs were also given in \cite{Bar01,LW02} under the stronger assumption that $f$ is smooth. The arguments of \cite{QianLi04,Bar01,LW02} rely on sign-changing critical point theorems built only for functionals of the form
\begin{equation*}
u\in H_0^1(\Omega)\mapsto \frac{1}{2}\|u\|^2-\Psi(u),
\end{equation*}
where $\Psi'$ is completely continuous, and cannot then be applied to \eqref{s} when $b>0$. Hence our result in Theorem \ref{mainresult} can be regarded as an extension of the classical result for the semilinear problem \eqref{classic} to the case of the nonlinear Kirchhoff type problem \eqref{s}. We also mention here that the result of Theorem \ref{mainresult} was more or less expected. However, it seems that this paper is the first to provide a formal proof. Moreover, we believe that the critical point theorem we will establish in the next section is of independant interest and can be applied to many other nonlocal problems (indeed, some applications by the author and collaborators will appear in other journals).

The study of sign-changing solutions is related to several long-standing questions concerning the multiplicity of solutions for elliptic boundary value problems. Compared with positive and negative solutions, sign-changing solutions have more complicated qualitative properties and are more difficult to find. During the last thirty years, several sophisticated techniques in calculus of variations and in critical point theory  were developed to study the multiplicity of sign-changing solutions to nonlinear elliptic partial differential equations. In \cite{Bar01} and \cite{LW02}, the authors established some multiplicity sign-changing critical point theorems in partially ordered Hilbert spaces by using Morse theory and the method of invariant sets of descending flow respectively. In \cite{Zou06}, a parameter-depending sign-changing fountain theorem was established without any Palais-Smale type assumption. More recently, a symmetric mountain pass theorem in the presence of invariant sets of the gradient flows was  introduced in \cite{LLW}.  However, it seems that all these powerful approaches are not directly applicable to find multiple sign-changing solutions to \eqref{s}.

 Our approach in proving Theorem \ref{mainresult} relies on a new sign-changing critical point theorem, also established in this paper, which is modelled on the fountain theorem of Bartsch \big(see Theorem 3.6 in \cite{W}\big). An essential tool in the proof of this theorem is a deformation lemma, which allows to lower sub-level sets of a functional, away from its critical set. The main ingredient in the proof of the deformation lemma is a suitable negative pseudo-gradient flow, a notion introduced by Palais \cite{Palais}. Since we are interesting in sign-changing critical points, the pseudo-gradient flow must be constructed in such the way that it keeps the positive and negative cones invariant. This invariance property makes the construction of the flow very complicated when the problem contains  nonlocal terms. In this paper, we borrow some ideas from recent work by Liu, Liu and Wang \cite{LLW} on the nonlinear Sch\"{o}dinger systems and by Liu, Wang and Zhang \cite{LZW} on the nonlinear Sch\"{o}dinger-Poisson system, where the pseudo-gradient flows were constructed by using an auxiliary operator. However, the critical point theorem used in \cite{LLW,LZW} cannot be applied to prove Theorem \ref{mainresult} because the corresponding auxiliary operator in the case of \eqref{s} is not compact. 

 The rest of the paper is organized as follows. In Section \ref{deux}, we state and prove the new sign-changing critical point theorem. In Section \ref{trois}, we provide the proof of Theorem \ref{mainresult}.

 Throughout the paper, we denote by $"\to"$ the strong converge and by $"\rightharpoonup"$ the weak convergence.
\section{An abstract sign-changing critical point theorem for even functionals}\label{deux}
In this section, we present a variant of the symmetric mountain pass type theorem which produces a sequence of sign-changing critical points with arbitrary large energy. 

 Let $\Phi$ be a $C^1$-functional defined on a Hilbert space $X$ of the form
\begin{equation}\label{space}
X:=\overline{\oplus_{j=0}^\infty X_j},\quad\text{with } \dim X_j<\infty.
\end{equation}
 We introduce for $k\geq2$ and $m>k+2$ the following notations:
\begin{equation*}
Y_k:=\oplus_{j=0}^k X_j,\quad Z_k=\overline{\oplus_{j=k}^\infty X_j},\quad Z^m_k=\oplus_{j=k}^m X_j,\quad  B_k:=\big\{u\in Y_k\,|\, \|u\|\leq\rho_k\big\}, 
\end{equation*}
\begin{equation*}
N_k:=\big\{u\in Z_k\,|\,\|u\|=r_k\big\},\,\, N^m_k:=\big\{u\in Z^m_k\,|\,\|u\|=r_k\big\},\,\,\text{where }0<r_k<\rho_k,
\end{equation*}
\begin{equation*}
\Phi_m:=\Phi|_{Y_m},\, K_m:=\big\{u\in Y_m\,;\, \Phi'_m(u)=0\big\}\,\text{ and }E_m:=Y_m\backslash K_m.
\end{equation*}
Let $P_m$ be a closed convex cone of $Y_m$. We set for $\mu_m>0$
\begin{equation*}
\pm D_m^0:=\big\{u\in Y_m\,|\, dist\big(u,\pm P_m\big)<\mu_m\big\},\,\,  D_m=D_m^0\cup(-D_m^0),\,\text{and } S_m:=Y_m\backslash D_m.
\end{equation*}
We will also denote the $\alpha$-neighborhood of $S\subset Y_m$ by
\begin{equation*}
V_\alpha(S):=\big\{u\in Y_m\,|\, dist(u,S)\leq\alpha\big\},\quad\forall\alpha>0.
\end{equation*}
Let us now state our critical point theorem. It is a version of the symmetric mountain pass theorem of Ambrosetti and Rabinowitz \cite{A-R}, and we model it on the fountain theorem of Bartsch \cite{B}.
\begin{theo}[Sign-changing fountain theorem]\label{scft}
Let $\Phi\in C^1(X,\mathbb{R})$ be an even functional which maps bounded sets to bounded sets. If, for $k\geq2$ and $m>k+2$, there exist $0<r_k<\rho_k$ and $\mu_m>0$ such that
\begin{enumerate}
\item[$(A_1)$] $a_k:=\max_{\substack{u\in \partial B_k}}\Phi(u)\leq0$ and $b_k:=\inf_{\substack{u\in N_k}}\Phi(u)\to+\infty$, as $k\to\infty$.
\item[$(A_2)$] $N^m_k\subset S_m$.
\item[$(A_3)$] There exists an odd locally Lipschitz continuous vector field $B:E_m\to Y_m$ such that:
\begin{itemize}
\item[$(i)$] $B\big((\pm D_m^0)\cap E_m\big)\subset \pm D_m^0$;
\item[$(ii)$] there exists a constant $\alpha_1>0$ such that $\big<\Phi'_m(u),u-B(u)\big>\geq\alpha_1\|u-B(u)\|^2$, for any $u\in E_m$;
\item[$(iii)$] for $a<b$ and $\alpha>0$, there exists $\beta>0$ such that $\|u-B(u)\|\geq\beta$ if $u\in Y_m$ is such that $\Phi_m(u)\in[a,b]$ and $\|\Phi'_m(u)\|\geq\alpha$.
\end{itemize}
\item[$(A_4)$] $\Phi$ satisfies the $(PS)^\star_{nod}$ condition, that is: 
\begin{itemize}
\item[$(i)$] any Palais-Smale sequence of $\Phi_m$ is bounded;
\item[$(ii)$] any sequence $(u_{m_j})\subset X$ such that
\begin{equation*}
m_j\to\infty,\quad u_{m_j}\in V_{\mu_{m_j}}(S_{m_j}),\quad \sup\Phi(u_{m_j})<\infty, \quad \Phi'_{m_j}(u_{m_j})=0
\end{equation*}
has a subsequence converging to a sign-changing critical point of $\Phi$.
\end{itemize}
\end{enumerate}
Then $\Phi$ has a sequence $(u_k)_k$ of sign-changing critical points in $X$ such that $\Phi(u_k)\to\infty$, as $k\to\infty$.
\end{theo}
Condition $(A_4)$ in Theorem \ref{scft} is a version of the usual compactness condition in Critical Point Theory, namely the Palais-Smale condition. We recall that a sequence $(u_n)\subset E$ is a Palais-Smale sequence of a smooth functional $J$ defined on a Banach space $E$ if the sequence \big($J(u_n)\big)$ is bounded and $J'(u_n)\to0$, as $n\to\infty$. If every such sequence possesses a convergent subsequence, then $J$ is said to satisfy the Palais-Smale condition.

 We need a special deformation lemma in order to prove the above result. We first recall the following helpful lemma.
\begin{lem}[\cite{Zou06}, Lemma 2.2]\label{helpful}
Let $\mathcal{M}$ be a closed convex subset of a Banach space $E$. If $H:\mathcal{M}\to E$ is a locally Lipschitz continuous map such that
\begin{equation*}
\lim_{\substack{\beta\to0^+}}\frac{dist\big(u+\beta H(u),\mathcal{M}\big)}{\beta}=0,\quad\forall u\in\mathcal{M},
\end{equation*}
then for any $u_0\in\mathcal{M}$, there exists $\delta>0$ such that the initial value problem
\begin{equation*}
\frac{d\sigma(t,u_0)}{dt}=H\big(\sigma(t,u_0)\big),\quad \sigma(0,u)=u_0,
\end{equation*}
has a unique solution defined on $[0,\delta)$. Moreover, $\sigma(t,u_0)\in\mathcal{M}$ for all $t\in[0,\delta)$.
\end{lem}
Now we state a quantitative deformation lemma.
\begin{lem}[Deformation lemma]\label{defor}
Let $\Phi\in C^1(X,\mathbb{R})$ be an even functional which maps bounded sets to bounded sets. Fix $m$ sufficiently large and assume that the condition $(A_3)$ of Theorem \ref{scft} holds. Let $c\in\mathbb{R}$ and $\varepsilon_0>0$ such that
\begin{equation}\label{one}
\forall u\in\Phi_m^{-1}\big([c-2\varepsilon_0,c+2\varepsilon_0]\big)\cap V_{\frac{\mu_m}{2}}(S_m)\,:\, \|\Phi_m'(u)\|\geq\varepsilon_0.
\end{equation}
Then for some $\varepsilon\in]0,\varepsilon_0[$ there exists $\eta\in C\big([0,1]\times Y_m,Y_m\big)$ such that:
\begin{enumerate}
\item[(i)] $\eta(t,u)=u$ for $t=0$ or $u\notin \Phi_m^{-1}\big([c-2\varepsilon,c+2\varepsilon]\big)$;
\item[(ii)] $\eta\big(1,\Phi_m^{-1}(]-\infty,c+\varepsilon])\cap S_m\big)\subset \Phi_m^{-1}\big(]-\infty,c-\varepsilon]\big)$;
\item[(iii)] $\Phi_m\big(\eta(\cdot,u)\big)$ is not increasing, for any $u$;
\item[(iv)] $\eta([0,1]\times D_m)\subset D_m$;
\item[(v)] $\eta(t,\cdot)$ is odd, for any $t\in[0,1]$.
\end{enumerate} 
\end{lem}
\begin{proof}
Define $V:E_m\to Y_m$ by $V(u)=u-B(u)$, where $B$ is given by $(A_3)$. Then there is $\delta>0$ such that $V(u)\geq\delta$ for any $u\in\Phi_m^{-1}\big([c-2\varepsilon_0,c+2\varepsilon_0]\big)\cap V_{\frac{\mu_m}{2}}(S_m)$ \big(in view $(A_3)$-(iii)\big). We take $\varepsilon\in ]0,\min(\varepsilon_0,\frac{\delta\alpha_1\mu_m}{8})[$ and we define
\begin{equation*}
A_1:=\Phi_m^{-1}\big([c-2\varepsilon,c+2\varepsilon]\big)\cap V_{\frac{\mu_m}{2}}(S_m),\quad A_2:=\Phi_m^{-1}\big([c-\varepsilon,c+\varepsilon]\big)\cap V_{\frac{\mu_m}{4}}(S_m),
\end{equation*}
\begin{equation*}
\chi(u):=\frac{dist(u,Y_m\backslash A_1)}{dist(u,Y_m\backslash A_1)+dist(u,A_2)}, \quad u\in Y_m
\end{equation*}
so that $\chi=0$ on $Y_m\backslash A_1$, $\chi=1$ on $A_2$, and $0\leq\chi\leq1$.\\
Consider the vector field
\begin{equation*}
W(u):=\left\{
      \begin{array}{ll}
        \chi(u)\|V(u)\|^{-2}V(u), \quad\text{ for }u\in  A_1 & \hbox{} \\
        \qquad \quad 0, \quad\textnormal{ for }u\in Y_m\backslash A_1. & \hbox{}
      \end{array}
    \right.
\end{equation*}
Clearly $W$ is odd and locally Lipschitz continuous. Moreover, by our choice of $\varepsilon$ above we have 
\begin{equation}\label{fourr}
\|W(u)\|\leq\frac{1}{\delta}\leq\frac{\alpha_1\mu_m}{8\varepsilon},\quad\forall u\in Y_m.
\end{equation}
 It follows that the Cauchy problem
\begin{equation*}
\frac{d}{dt}\sigma(t,u)=-W(\sigma(t,u)),\qquad\sigma(0,u)=u\in Y_m
\end{equation*}
has a unique solution $\sigma(\cdot, u)$ defined on $\mathbb{R}_+$. Moreover, $\sigma$ is continuous on $\mathbb{R}_+\times X$.\\
We have in view of \eqref{fourr}
\begin{equation}\label{two}
\|\sigma(t,u)-u\|\leq \int_0^t\|W(\sigma(s,u))\|ds\leq\frac{\alpha_1\mu_m}{8\varepsilon}t,
\end{equation}
and by $(A_3)$-(ii) 
\begin{align}
\nonumber \frac{d}{dt}\Phi_m(\sigma(t,u))&=-\big<\Phi_m'(\sigma(t,u)),\chi(\sigma(t,u))\|V(\sigma(t,u))\|^{-2}V(\sigma(t,u))\big>\\
& \leq -\alpha_1\chi(\sigma(t,u)).\label{three}
\end{align}
Define 
\begin{equation*}
\eta:[0,1]\times Y_m\to Y_m,\qquad \eta(t,u):=\sigma\big(\frac{2\varepsilon}{\alpha_1}t,u\big).
\end{equation*}
Conclusion (i) of the lemma is clearly satisfied and by \eqref{three} above (iii) is also satisfied. Since $W$ is odd, (v) is a consequence of the uniqueness of the solution to the above Cauchy problem.

 We now verify (ii). Let $v\in\eta(1,\Phi_m^{-1}(]-\infty,c+\varepsilon])\cap S_m)$. Then $v=\eta(1,u)=\sigma(\frac{2\varepsilon}{\alpha_1},u)$, where $u\in \Phi_m^{-1}(]-\infty,c+\varepsilon])\cap S_m$. \\
If there exists $t\in[0,\frac{2\varepsilon}{\alpha_1}]$ such that $\Phi_m(\sigma(t,u))<c-\varepsilon$, then by (iii) we have $\Phi_m(v)<c-\varepsilon$.\\
Assume now that $\sigma(t,u)\in\Phi_m^{-1}([c-\varepsilon,c+\varepsilon])$ for all $t\in[0,\frac{2\varepsilon}{\alpha_1}]$. By \eqref{two} we have $\|\sigma(t,u)-u\|\leq\frac{\mu_m}{4}$, which means, since $u\in S_m$, that $\sigma(t,u)\in V_{\frac{\mu_m}{4}}(S_m)$. Hence $\sigma(t,u)\in A_2$ and since $\chi=1$ on $A_2$, we deduce from \eqref{three} that
\begin{equation*}
\Phi_m\big(\sigma(\frac{2\varepsilon}{\alpha_1},u)\big)\leq \Phi_m(u)-\alpha_1\int_0^{\frac{2\varepsilon}{\alpha_1}}\chi(\sigma(t,u))dt=\Phi_m(u)-2\varepsilon.
\end{equation*}
This implies, since $\Phi_m(u)\leq c+\varepsilon$, that $\Phi_m(v)=\Phi_m(\sigma(\frac{2\varepsilon}{\alpha_1},u))\leq c-\varepsilon$. Hence (ii) is satisfied.

 It remains to verify (iv). Since $\sigma$ is odd in $u$, it suffices to show that
\begin{equation}\label{seven}
\sigma\big([0,+\infty)\times D_m^0\big)\subset D_m^0.
\end{equation}
We follow \cite{Zou06}.

\textbf{Claim:} We have
\begin{equation}\label{eight}
\sigma([0,+\infty)\times \overline{D_m^0})\subset \overline{D_m^0}.
\end{equation}
Assume by contradiction  that \eqref{seven} does not hold. Then there exist $u_0\in D_m^0$ and $t_0>0$ such that $\sigma(t_0,u_0)\notin D_m^0$. Choose a neighborhood $N_{u_0}$ of $u_0$ such that $N_{u_0}\subset D_m^0$. Then there is a neighborhood $N_0$ of $\sigma(t_0,u_0)$ such that $\sigma(t_0,\cdot):N_{u_0}\to N_0$ is a homeomorphism. Since $\sigma(t_0,u_0)\notin D_m^0$, the set $N_0\backslash \overline{D_m^0}$ is not empty. Hence there is $w\in N_{u_0}$ such that $\sigma(t_0,w)\in N_0\backslash \overline{D_m^0}$, contradicting \eqref{eight}.\\
We now terminate by giving the proof of our above claim.\\
By $(A_3)$-(i) we have $B( D_m^0\cap E_m)\subset D_m^0$, which implies that $B( \overline{D_m^0}\cap E_m)\subset  \overline{D_m^0}$. Since $K_m\cap A_1=\emptyset$, we have $\sigma(t,u)=u$ for all $t\in[0,1]$ and $u\in \overline{D^0_m}\cap K_m$.\\
Assume that $u\in \overline{D^0_m}\cap E_m$. If there is $t_1\in(0,1]$ such that $\sigma(t_1,u)\notin \overline{D_m^0}$, then there would be $s_1\in[0,t_1)$ such that $\sigma(s_1,u)\in\partial\overline{D_m^0}$ and $\sigma(t,u)\notin \overline{D_m^0}$ for all $t\in(0,t_1]$. The following Cauchy problem
\begin{equation*}
\frac{d}{dt}\mu(t,\sigma(s_1,u))=-W\big(\mu(t,\sigma(s_1,u))\big),\qquad \mu(0,\sigma(s_1,u))=\sigma(s_1,u)\in Y_m
\end{equation*}
has $\sigma(t,\sigma(s_1,u))$ as unique solution. Recalling that $W=0$ on $Y_m\backslash A_1$, we have $v-W(v)\in \overline{D_m^0}\cap (Y_m\backslash A_1)$ for any $v\in\overline{D_m^0}\cap(Y_m\backslash A_1)$.\\
Assume that $v\in A_1\cap \overline{D_m^0}$. Since $\|V(u)\|\geq\delta$, we deduce that $1-\beta\chi(v)\|V(v)\|^{-2}\geq0$ for all $\beta$ such that $0<\beta\leq \delta^2$. Recalling that $v\in\overline{D_m^0}$ implies $dist(v,P_m)\leq\mu_m$, that $V(v)=v-B(v)$, and that $aP_m+bP_m\subset P_m$ for all $a,b\geq0$ \big(because $P_m$ is a cone\big), we obtain for any $\beta\in]0,\delta^2]$
\begin{align*}
dist\big(v-\beta W(v),P_m\big)&=dist\big(v-\beta\chi(v)\|V(v)\|^{-2}V(v),P_m\big)\\
&=dist\big(v-\beta\chi\|V(v)\|^{-2}(v-B(v)),P_m\big)\\
&=dist\Big((1-\beta\chi(v)\|V(v)\|^{-2})v+\beta\chi(v)\|V(v)\|^{-2}B(v),P_m\Big)\\
&\leq dist\Big((1-\beta\chi(v)\|V(v)\|^{-2})v+\beta\chi(v)\|V(v)\|^{-2}B(v),\\
&\beta\chi(u)\|V(v)\|^{-2}P_m+(1-\beta\chi(v)\|V(v)\|^{-2})P_m\Big)\\
\leq (1&-\beta\chi(v)\|V(v)\|^{-2})dist(v,P_m)+\beta\chi(v)\|V(v)\|^{-2}dist(B(v),P_m)\\
&\leq (1-\beta\chi(v)\|V(v)\|^{-2})\mu_m+\beta\chi(v)\|V(v)\|^{-2}\mu_m\\
&=\mu_m.
\end{align*}
It follows that $v-\beta W(v)\in\overline{D_m^0}$ for $0<\beta\leq \delta^2$. This implies that
\begin{equation*}
\lim_{\substack{\beta\to0^+}}\frac{dist\big(v+\beta(- W(v)),\overline{D_m^0}\big)}{\beta}=0,\quad\forall u\in\overline{D_m^0}.
\end{equation*}
By Lemma \ref{helpful} there then exists $\delta_0>0$ such that $\sigma\big(t,\sigma(s_1,u)\big)\in \overline{D_m^0}$ for all $t\in[0,\delta_0)$. This implies that $\sigma\big(t,\sigma(s_1,u)\big)=\sigma(t+s_1,u)\in \overline{D_m^0}$ for all $t\in[0,\delta_0)$, which contradicts the definition of $s_1$. This last contradiction assures that $\sigma([0,+\infty)\times\overline{D_m^0})\subset\overline{D_m^0}$.
\end{proof}
\begin{proof}[Proof of Theorem \ref{scft}]
$(A_1)$ and $(A_2)$ imply that $a_k<b_k\leq\inf_{\substack{u\in N_k^m}}\Phi_m(u)$, for $k$ big enough. Let
\begin{equation*}
\Gamma_k^m:=\big\{\gamma\in C(B_k,Y_m)\,;\, \gamma \text{ is odd, } \gamma|_{\partial B_k}=id\text{ and }\gamma(D_m)\subset D_m\big\}.
\end{equation*}
$\Gamma_k^m$ is clearly non empty and for any $\gamma\in \Gamma_k^m$ the set $U:=\big\{u\in B_k\,;\,\|\gamma(u)\|<r_k\big\}$ is an open bounded and symmetric \big(i.e. $-U=U$\big) neighborhood of the origin in $Y_k$. By the Borsuk-Ulam theorem the continuous odd map $\Pi_k\circ\gamma:\partial U\subset Y_k\to Y_{k-1}$ has a zero, where $\Pi_k:X\to Y_{k-1}$ is the orthogonal projection. It then follows that $\gamma(B_k)\cap N_k^m\neq\emptyset$ and, since $N_k^m\subset S_m$, that $\gamma(B_k)\cap S_m\neq\emptyset$. This intersection property implies that
\begin{equation*}
c_{k,m}:=\inf_{\gamma\in\Gamma_k^m}\max_{u\in\gamma(B_k)\cap S_m}\Phi_m(u)\geq \inf_{\substack{u\in N^m_k}}\Phi(u)\geq b_k.
\end{equation*}
We would like to show that for any $\varepsilon_0\in]0,\frac{c_{k,m}-a_k}{2}[$, there exists $u\in\Phi_m^{-1}\big([c_{k,m}-2\varepsilon_0,c_{k,m}+2\varepsilon_0]\big)\cap V_{\frac{\mu_m}{2}}(S_m)$ such that $\|\Phi'_m(u)\|<\varepsilon_0$.\\
Arguing by contradiction, we assume that we can find $\varepsilon_0\in]0,\frac{c_{k,m}-a_k}{2}[$ such that 
\begin{equation*}
\|\Phi'_m(u)\|\geq\varepsilon_0,\quad\forall u\in\Phi_m^{-1}\big([c_{k,m}-2\varepsilon_0,c_{k,m}+2\varepsilon_0]\big)\cap V_{\frac{\mu_m}{2}}(S_m).
\end{equation*}
Apply Lemma \ref{defor} with $c=c_{k,m}$ and define, using the deformation $\eta$ obtained, the map
\begin{equation*}
\theta:B_k\to Y_m,\qquad \theta(u):=\eta(1,\gamma(u)),
\end{equation*}
where $\gamma\in\Gamma_k^m$ satisfies
\begin{equation}\label{six}
\max_{\substack{u\in\gamma(B_k)\cap S_m}}\Phi_m(u)\leq c_{k,m}+\varepsilon,
\end{equation}
with $\varepsilon$ also given by Lemma \ref{defor}.\\
Using the properties of $\eta$ (see Lemma \ref{defor}), one can easily verify that $\theta\in\Gamma^m_k$.\\
On the other hand, we have
\begin{equation}\label{five}
\eta\big(1,\gamma(B_k)\big)\cap S_m\subset \eta\big(1,\Phi_m^{-1}(]-\infty,c_{k,m}+\varepsilon])\cap S_m\big).
\end{equation}
In fact, if $u\in \eta\big(1,\gamma(B_k)\big)\cap S_m$ then $u=\eta\big(1,\gamma(v)\big)\in S_m$ for some $v\in B_k$. Observe that $\gamma(v)\in S_m$. Indeed, if this is not true then $\gamma(v)\in D_m$, and by (iv) of Lemma \ref{defor} we obtain $u=\eta(1,\gamma(v))\in D_m$ which contradicts the fact that $u\in S_m$. Now \eqref{six} implies that $\gamma(v)\in\Phi_m^{-1}(]-\infty,c_{k,m}+\varepsilon])$. It then follows, using (ii) of Lemma \ref{defor}, that $u=\eta(1,\gamma(v))\in\eta\big(1,]-\infty,c_{k,m}+\varepsilon]\cap S_m\big)$. Hence \eqref{five} holds.\\
Using \eqref{five} and (ii) of Lemma \ref{defor}, we obtain
\begin{align*}
\max_{\substack{u\in\theta(B_k)\cap S_m}}\Phi_m(u)&=\max_{\substack{u\in\eta\big(1,\gamma(B_k)\big)\cap S_m}}\Phi_m(u)\\
&\leq \max_{\substack{u\in\eta\big(1,\Phi_m^{-1}(]-\infty,c_{k,m}+\varepsilon])\cap S_m\big)}}\Phi_m(u)\\
&\leq c_{k,m}-\varepsilon,
\end{align*}
contradicting the definition of $c_{k,m}$.\\
The above contradiction assures that for any $\varepsilon_0\in]0,\frac{c_{k,m}-a_k}{2}[$, there exists $u\in\Phi_m^{-1}\big([c_{k,m}-2\varepsilon_0,c_{k,m}+2\varepsilon_0]\big)\cap V_{\frac{\mu_m}{2}}(S_m)$ such that $\|\Phi'_m(u)\|<\varepsilon_0$.\\
It follows that there is a sequence $(u^n_{k,m})_n\subset V_{\frac{\mu_m}{2}}(S_m)$ such that
\begin{equation*}
\Phi'_m(u^n_{k,m})\to0\text{ and }\Phi_m(u^n_{k,m})\to c_{k,m},\text{ as }n\to\infty.
\end{equation*}
We deduce from $(A_4)$-(i), using the fact $Y_m$ is finite-dimensional, that there exists $u_{k,m}\in V_{\frac{\mu_m}{2}}(S_m)$ such that
\begin{equation*}
 \Phi'_m(u_{k,m})=0\text{ and }\Phi_m(u_{k,m})=c_{k,m}.
\end{equation*}
Noting that $c_{k,m}\leq\max_{u\in B_k}\Phi(u)$, we deduce using $(A_4)$-(ii) that $\Phi$ has a sign-changing critical point $u_k$ such that $b_k\leq\Phi(u_k)\leq \max_{u\in B_k}\Phi(u)$. Since $b_k\to\infty$, as $k\to\infty$, the conclusion follows.
\end{proof}
\section{Proof of the main result}\label{trois}
Throughout this section, we assume that $(f_{1,2,3,4})$ are satisfied. We denote by $|\cdot|_q$ the usual norm of the Lebesgue space $L^q(\Omega)$.

 Let $X:=H_0^1(\Omega)$ be the usual Sobolev space endowed with the inner product
\begin{equation*}
\big<u,v\big>=\int_\Omega\nabla u\nabla vdx
\end{equation*}
and norm $\|u\|^2=\big<u,u\big>$, for $u,v\in H_0^1(\Omega)$.

 It is well known that solutions of \eqref{s} are critical points of the functional 
\begin{equation}
\Phi(u)=\frac{a}{2}\|u\|^2+\frac{b}{4}\|u\|^4-\int_\Omega F(x,u)dx,\quad u\in X:=H_0^1(\Omega).
\end{equation}
By a standard argument, one can easily verify that $\Phi$ is of class $C^1$ and 
\begin{equation}
\big<\Phi'(u),v\big>=\big(a+b\|u\|^2\big)\int_\Omega\nabla u\nabla vdx-\int_\Omega vf(x,u)dx
\end{equation}

 Let $0<\lambda_1<\lambda_2<\lambda_3<\cdots$ be the distinct eigenvalues of the problem
\begin{equation*}
-\Delta u=\lambda u\quad\text{in }\Omega,\qquad u=0\quad\text{on }\partial\Omega.
\end{equation*}
 Then each $\lambda_j$ has finite multiplicity. It is well known that the principal eigenvalue $\lambda_1$ is simple with a positive eigenfunction $e_1$, and the eigenfunctions $e_j$ corresponding to $\lambda_j$ ($j\geq2$) are sign-changing. Let $X_j$ be the eigenspace associated to $\lambda_j$. We set for $k\geq2$
\begin{equation*}
Y_k:=\oplus_{j=1}^k X_j\text{ and } Z_k=\overline{\oplus_{j=k}^\infty X_j}. 
\end{equation*}
\begin{lem}\label{akbk}\quad
\begin{enumerate}
\item[(1)] For any $u\in Y_k$ we have $\Phi(u)\to-\infty$, as $\|u\|\to\infty$.
\item[(2)] There exists $r_k>0$ such that 
\begin{equation*}
\inf_{\substack{u\in Z_k\\\|u\|=r_k}}\Phi(u)\to\infty,\text{ as }k\to\infty.
\end{equation*}
\end{enumerate}
\end{lem}
\begin{proof}
(1)\quad It is well known that integrating $(f_3)$ yields the existence of two constants $c_1,c_2>0$ such that $F(x,u)\geq c_1|u|^\mu-c_2$. This together with the fact that all norms are equivalent in the finite-dimensional subspace $Y_k$ imply that
\begin{equation*}
\Phi(u)\leq \frac{a}{2}\|u\|^2+\frac{b}{4}\|u\|^4-c_3\|u\|^\mu+c_4,\quad\forall u\in Y_k,
\end{equation*}
where $c_3,c_4>0$ are constant. Since $\mu>4$, it follows that $\Phi(u)\to-\infty$, as $\|u\|\to\infty$. \\
(2)\quad Using $(f_1)$, we obtain
\begin{equation*}
\Phi(u)\geq \frac{a}{2}\|u\|^2-c_5|u|_p^p-c_6, \quad\forall u\in X,
\end{equation*}
where $c_5,c_6>0$ are constant. Set
\begin{equation*}
\beta_k:=\sup_{\substack{v\in Z_k\\\|v\|=1}}|v|_p.
\end{equation*}
 Then we obtain
 \begin{equation*}
 \Phi(u)\geq a\big(\frac{1}{2}-\frac{1}{p}\big)\big(\frac{c_5}{a}p\beta_k^p\big)^{\frac{2}{2-p}}-c_6
 \end{equation*}
 for every $u\in Z_k$ such that
\begin{equation*}
\|u\|=r_k:=\big(\frac{c_5}{a}p\beta_k^p\big)^{\frac{1}{2-p}}.
\end{equation*} 
We know from Lemma $3.8$ in \cite{W} that $\beta_k\to0$, as $k\to\infty$. This implies that $r_k\to\infty$, as $k\to\infty$.
\end{proof}
Now we fix $k$ large enough and we set for $m>k+2$
\begin{equation*}
\Phi_m:=\Phi|_{Y_m},\, K_m:=\big\{u\in Y_m\,;\, \Phi'_m(u)=0\big\},\, E_m:=Y_m\backslash K_m,
\end{equation*}
\begin{equation*}
P_m:=\big\{u\in Y_m\,;\,u(x)\geq0\big\},\,\, Z^m_k:=\oplus_{j=k}^m X_j,\,\text{ and }  N^m_k:=\big\{u\in Z^m_k\,|\,\|u\|=r_k\big\}.
\end{equation*}
Remark that for all $u\in P_m\backslash\big\{0\big\}$ we have $\int_\Omega ue_1dx>0$, while for all $u\in Z_k$, $\int_\Omega ue_1dx=0$, where $e_1$ is the principal eigenfunction of the Laplacian. This implies that $P_m\cap Z_k=\big\{0\big\}$. It then follows, since $N_k^m$ is compact, that
\begin{equation}\label{deltam}
\delta_m:=dist\big(N_k^m,-P_m\cup P_m\big)>0.
\end{equation}
For $u\in Y_m$ fixed, we consider the functional
\begin{equation}
I_u(v)=\frac{1}{2}\big(a+b\|u\|^2\big)\int_\Omega|\nabla v|^2dx-\int_\Omega vf(x,u)dx,\quad u\in Y_m.
\end{equation}
It is not difficult to see that $I_u$ is of class $C^1$, coercive, bounded below, weakly lower semicontinuous, and strictly convex. Therefore $I_u$ admits a unique minimizer $v=Au\in Y_m$, which is the unique solution to the problem
\begin{equation*}
-\big(a+b\|u\|^2\big)\Delta v=f(x,u),\quad v\in Y_m.
\end{equation*} 
Clearly, the set of fixed points of $A$ coincide with $K_m$. Moreover, the operator $A:Y_m\to Y_m$ has the following important properties.
\begin{lem}\label{Aop}\quad
\begin{enumerate}
\item[(1)] $A$ is continuous and maps bounded sets to bounded sets.
\item[(2)] For any $u\in Y_m$ we have
\begin{align}
\big<\Phi_m'(u),u-Au\big>\geq a\|u-Au\|^2, \label{aa1}\\
\|\Phi_m'(u)\|\leq (a+b)\big(1+\|u\|^2\big)\|u-Au\|.\label{aa2}
\end{align}
\item[(3)] There exists $\mu_m\in]0,\delta_m[$ such that $A(\pm D^0_m)\subset \pm D^0_m$, where $\delta_m$ is defined by \eqref{deltam}.
\end{enumerate}
\end{lem}
\begin{proof}
(1)\quad Let $(u_n)\subset Y_m$ such that $u_n\to u$. We set $v_n=Au_n$ and $v=Au$.\\
By the definition of $A$ we have for any $w\in Y_m$
\begin{align}
\big(a+b\|u_n\|^2\big)\int_\Omega\nabla v_n\nabla wdx&=\int_\Omega wf(x,u_n)dx\label{e1}\\
\big(a+b\|u\|^2\big)\int_\Omega\nabla v\nabla wdx&=\int_\Omega wf(x,u)dx.\label{e2}
\end{align}
Taking $w=v_n-v$ in \eqref{e1} and in \eqref{e2}, and using the H\"{o}lder inequality and the Sobolev embedding theorem, we obtain
\begin{align*}
\big(a+b\|u_n\|^2\big)\|v_n-v\|^2&=b\big(\|u_n\|^2-\|u\|^2\big)\int_\Omega\nabla v\nabla(v_n-v)dx\\
&\qquad\qquad\qquad\qquad\qquad+\int_\Omega(v-v_n)\big(f(u_n)-f(u)\big)dx\\
&\leq c_1\big|\|u_n\|^2-\|u\|^2\big|\|v\|\|v_n-v\|+c_2\|v_n-v\||f(u_n)-f(u)|_{\frac{p}{p-1}},
\end{align*}
where $c_1,c_2>0$ are constant. By $(f_1)$ and Theorem $A.2$ in \cite{W}, we have $f(u_n)- f(u)\to0$ in $L^{\frac{p}{p-1}}(\Omega)$. Hence $\|Au_n-Au\|=\|v_n-v\|\to0$, that is, $A$ is continuous.

 On the other hand, for any $u\in Y_m$ we have, taking $v=w=Au$ in \eqref{e2}
\begin{equation*}
\big(a+b\|u\|^2\big)\|Au\|^2=\int_\Omega Auf(x,u)dx.
\end{equation*}
By using $(f_1)$, the H\"{o}lder inequality, and the Sobolev embedding theorem, we obtain
\begin{equation*}
a\|Au\|\leq C\big(1+\|u\|^{p-1}),
\end{equation*}
where $C>0$ a constant. This shows that $Au$ is bounded whenever $u$ is bounded.\\
(2)\quad Taking $w=u-Au$ in \eqref{e2}, we obtain
\begin{equation*}
\big(a+b\|u\|^2\big)\int_\Omega\nabla(Au)\nabla(u-Au)dx=\int_\Omega(u-Au)f(x,u)dx,
\end{equation*}
which implies that
\begin{equation*}
\big<\Phi_m'(u),u-Au\big>=\big(a+b\|u\|^2\big)\|u-Au\|^2\geq a\|u-Au\|^2.
\end{equation*}
On the other hand, we obtain using \eqref{e2}
\begin{align*}
\big<\Phi'_m(u),w\big>&=\big(a+b\|u\|^2\big)\int_\Omega\nabla u\nabla wdx-\int_\Omega wf(x,u)dx\\
&=\big(a+b\|u\|^2\big)\int_\Omega\nabla (u-Au)\nabla wdx,\quad\forall w\in Y_m.
\end{align*}
This implies that 
\begin{equation*}
\|\Phi'_m(u)\|\leq \big(a+b\|u\|^2\big)\|u-Au\|.
\end{equation*}
(3)\quad It follows from $(f_1)$ and $(f_2)$ that
\begin{equation}\label{feps}
\forall\varepsilon>0,\quad\exists c_\varepsilon>0\,\,;\,\, |f(x,t)|\leq \varepsilon|t|+c_\varepsilon|t|^{p-1},\quad\forall t\in\mathbb{R}.
\end{equation}
Let $u\in Y_m$ and let $v=Au$. As usual we denote $w^\pm=\max\{0,\pm w\}$, for any $w\in X$.\\
Taking $w=v^+$ in \eqref{e2} and using the H\"{o}lder inequality, we obtain
\begin{equation*}
\big(a+b\|u\|^2\big)\|v^+\|^2=\int_\Omega v^+f(x,u)dx\leq \varepsilon|u^+|_2|v^+|_2+c_\varepsilon|u^+|_p^{p-1}|v^+|_p,
\end{equation*}
which implies that
\begin{equation}\label{v}
\|v^+\|^2\leq \frac{1}{a}\Big(\varepsilon|u^+|_2|v^+|_2+c_\varepsilon|u^+|_p^{p-1}|v^+|_p\Big).
\end{equation}
On the other hand it is not difficult to see that $|u^+|_q\leq |u-w|_q$ for all $w\in -P_m$ and $1\leq q\leq 2^\star$. Hence there is a constant $c_1=c_1(q)>0$ such that $|u^+|_q\leq c_1 dist(u,-P_m)$. It is obvious that $dist(v,-P_m)\leq \|v^+\|$. So we deduce from \eqref{v} and the Sobolev embedding theorem that
\begin{align*}
dist(v,-P_m)\|v^+\|&\leq \|v^+\|^2\\
&\leq c_2\Big(\varepsilon dist(u,-P_m)+c_\varepsilon dist(u,-P_m)^{p-1}\Big)\|v^+\|,
\end{align*}
where $c_2>0$ is constant. This implies that 
\begin{equation*}
dist(v,-P_m)\leq  c_2\Big(\varepsilon dist(u,-P_m)+c_\varepsilon dist(u,-P_m)^{p-1}\Big).
\end{equation*}
Similarly on can show that
\begin{equation*}
dist(v,P_m)\leq  c_3\Big(\varepsilon dist(u,P_m)+c_\varepsilon dist(u,P_m)^{p-1}\Big),
\end{equation*}
for some constant $c_3>0$.\\
Choosing $\varepsilon$ small enough, we can then find $\mu_m\in]0,\delta_m[$ such that
\begin{equation*}
dist(v,\pm P_m)\leq \frac{1}{2} dist(u,\pm P_m)
\end{equation*}
 whenever $dist(u,\pm P_m)<\mu_m$.
\end{proof}
Using $\mu_m$ obtained above, we define
\begin{equation*}
\pm D_m^0:=\big\{u\in Y_m\,|\, dist\big(u,\pm P_m\big)<\mu_m\big\},\,\,  D_m=D_m^0\cup(-D_m^0)
\end{equation*}
\begin{equation*}
 S_m:=Y_m\backslash D_m.
\end{equation*}
\begin{rem}\label{a2}
$\mu_m<\delta_m\Longrightarrow N_k^m\subset S_m.$
\end{rem}
The vector field $A:Y_m\to Y_m$ does not satisfy the assumption $(A_3)$ of Theorem \ref{scft} as it is not locally Liptschitz continuous. However, it will be used in the spirit of \cite{BL04} to construct a vector field which will satisfy the above mentioned condition. 
\begin{lem}\label{Bop}
There exists an odd locally Lipschitz continuous operator $B:E_m\to Y_m$ such that
\begin{enumerate}
\item[(1)] $\big<\Phi'(u),u-B(u)\big>\geq \frac{1}{2}\|u-A(u)\|^2$, for any $u\in E_m$.
\item[(2)] $\frac{1}{2}\|u-B(u)\|\leq \|u-A(u)\|\leq 2\|u-B(u)\|$, for any $u\in E_m$.
\item[(3)] $B\big((\pm D^0_m)\cap E_m\big)\subset \pm D^0_m$.
\end{enumerate} 
\end{lem}
The proof of this lemma follows the lines of \cite{BL04}. We provide a sketch of the proof here for completeness.
\begin{proof}
We define $\Delta_1,\Delta_2:E_m\to\mathbb{R}$ as
\begin{equation}\label{deltazero}
\Delta_1(u)=\frac{1}{2}\|u-Au\|\,\,\textnormal{ and }\,\,\Delta_2(u)=\frac{a}{2(a+b)}(1+\|u\|^2)^{-1}\|u-Au\|.
\end{equation}
For any $u\in E_m$ we choose $\gamma(u)>0$ such that 
\begin{equation}\label{deltaun}
\|A(v)-A(w)\|<\min\big\{\Delta_1(v),\Delta_1(w),\Delta_2(v),\Delta_2(w)\big\}
\end{equation}
holds for every $v,w\in N(u):=\big\{z\in Y_m\,;\,\|z-u\|<\gamma(u)\big\}$.\\
Let $\mathcal{V}$ be a locally finite open refinement of $\big\{N(u)\,;\,u\in E_m\big\}$ and define
\begin{equation*}
\mathcal{V}^\star:=\big\{V\in\mathcal{V}\,;\, D_m^0\cap V\neq\emptyset,\,-D_m^0\cap V\neq\emptyset,\,-D_m^0\cap D_m^0\cap V\neq\emptyset \big\},
\end{equation*}
\begin{equation*}
\mathcal{U}:=\bigcup\limits_{V\in\mathcal{V}\backslash\mathcal{V}^\star}\big\{V\big\}\cup\bigcup\limits_{V\in\mathcal{V}^\star}\big\{V\backslash D_m^0,V\backslash(-D_m^0)\big\}.
\end{equation*}
By construction $\mathcal{U}$ is a locally finite open refinement of $\big\{N(u)\,;\, u\in E_m\big\}$ and has a property that any $U\in\mathcal{U}$ is such that
\begin{equation}\label{etoile}
U\cap D_m^0\neq\emptyset\text{ and }U\cap (-D_m^0)\neq\emptyset\Longrightarrow U\cap D_m^0\cap (-D_m^0)\neq\emptyset.
\end{equation}
Let $\big\{\Pi_U\,:\,U\in\mathcal{U}\big\}$ be the partition of unity subordinated to $\mathcal{U}$ defined by
\begin{equation*}
\Pi_U(u):=\frac{\alpha_U(u)}{\sum\limits_{v\in\mathcal{U}}\alpha_U(v)},\text{ where }\alpha_U(u)=dist\big(u,E_m\backslash U\big).
\end{equation*}
For any $u\in\mathcal{U}$ choose $a_U$ such that if $U\cap(\pm D_m^0)\neq\emptyset$ then $a_U\in U\cap(\pm D_m^0)$ \big(such an element exists in view of \eqref{etoile}\big). Define $B:E_m\to Y_m$ by 
\begin{equation*}
B(u):=\frac{1}{2}\big(H(u)-H(-u)\big),\quad\text{where } H(u)=\sum\limits_{U\in\mathcal{U}}\Pi_U(u)A(a_U).
\end{equation*}
We then conclude as in \cite{BL04} by using Lemma \ref{Aop}-(3), \eqref{deltazero}, \eqref{deltaun}, and \eqref{aa1}.
\end{proof}
\begin{rem}\label{rema}
Lemmas \ref{Aop} and \ref{Bop} imply that 
\begin{align*}
\big<\Phi_m'(u),u-B(u)\big>&\geq\frac{1}{8}\|u-B(u)\|^2 \text{ and}\\
\|\Phi_m'(u)\|\leq2(a+b)(1&+\|u\|^2)\|u-B(u)\|,\text{ for all } u\in E_m.
\end{align*}
\end{rem}
\begin{lem}\label{tech}
Let $c<d$ and $\alpha>0$. For all $u\in Y_m$ such that $\Phi_m(u)\in[c,d]$ and $\|\Phi'_m(u)\|\geq\alpha$, there exists $\beta>0$ such that $\|u-B(u)\|\geq\beta$.
\end{lem}
\begin{proof}
By the definition of the operator $A$, we have for any $u\in Y_m$
\begin{equation*}
\big(a+b\|u\|^2\big)\int_\Omega\nabla(Au)\nabla u dx=\int_\Omega uf(x,u)dx.
\end{equation*}
It follows that
\begin{multline*}
\Phi_m(u)-\frac{1}{\mu}\big(a+b\|u\|^2\big)\int_\Omega\nabla u\nabla (u-Au) dx=a\big(\frac{1}{2}-\frac{1}{\mu}\big)\|u\|^2\\
+b\big(\frac{1}{4}-\frac{1}{\mu}\big)\|u\|^4+\int_\Omega\big(\frac{1}{\mu}uf(x,u)-F(x,u)\big)dx
\end{multline*}
which implies, using $(f_3)$ and Lemma \ref{Bop}-(2), that
\begin{align}
\nonumber b\big(\frac{1}{4}-\frac{1}{\mu}\big)\|u\|^4&\leq|\Phi_m(u)|+\frac{1}{\mu}\big(a+b\|u\|^2\big)\|u\|\|u-Au\|\\
&\leq |\Phi_m(u)|+\frac{2}{\mu}\big(a+b\|u\|^2\big)\|u\|\|u-Bu\|.\label{deuxetoile}
\end{align}
Suppose that there exists a sequence $(u_n)\subset Y_m$ such that $\Phi_m(u_n)\in[c,d]$, $\|\Phi_m'(u_n)\|\geq\alpha$ and $\|u_n-Bu_n\|\to0$. By \eqref{deuxetoile} we see that $(\|u_n\|)$ is bounded. It follows from Remark \ref{rema} above that $\Phi'_m(u_n)\to0$, which is a contradiction.
\end{proof}
Now we verify the compactness condition for $\Phi$.
\begin{lem}\label{psnod}
$\Phi$ satisfies the $(PS)^\star_{nod}$ condition, that is:
\begin{itemize}
\item any Palais-Smale sequence of $\Phi_m$ is bounded and
\item any sequence $(u_{m_j})\subset X$ such that
\begin{equation*}
m_j\to\infty,\quad u_{m_j}\in V_{\frac{\mu_{m_j}}{2}}(S_{m_j}),\quad \sup\Phi(u_{m_j})<\infty, \quad \Phi'_{m_j}(u_{m_j})=0
\end{equation*}
has a subsequence converging to a sign-changing critical point of $\Phi$.
\end{itemize}
\end{lem}
\begin{proof}
For any $u\in Y_m$ we have, in view of $(f_3)$, 
\begin{align}
\nonumber\Phi_m(u)-\frac{1}{\mu}\big<\Phi_m'(u),u\big>&=a\big(\frac{1}{2}-\frac{1}{\mu}\big)\|u\|^2+b\big(\frac{1}{4}-\frac{1}{\mu}\big)\|u\|^4\\
\nonumber&\qquad\qquad\qquad+\int_\Omega\big(\frac{1}{\mu}uf(x,u)-F(x,u)\big)dx\\
&\geq a\big(\frac{1}{2}-\frac{1}{\mu}\big)\|u\|^2+b\big(\frac{1}{4}-\frac{1}{\mu}\big)\|u\|^4.\label{pss}
\end{align}
It then follows that any sequence $(u_n)\subset Y_m$ such that $\sup_n\Phi_m(u_n)<\infty$ and $\Phi_m'(u_n)\to0$ is bounded.

 Now let $(u_{m_j})\subset X$ be such that 
\begin{equation*}
m_j\to\infty,\quad u_{m_j}\in V_{\frac{\mu_{m_j}}{2}}(S_{m_j}),\quad \sup\Phi(u_{m_j})<\infty, \quad \Phi'_{m_j}(u_{m_j})=0.
\end{equation*}
In view of \eqref{pss} the sequence $(u_{m_j})$ is bounded. Hence, up to a subsequence, $u_{m_j}\rightharpoonup u$ in $X$ and $u_{m_j}\to u$ in $L^p(\Omega)$.\\
Observe that the condition $\Phi_{m_j}'(u_{m_j})=0$ is weaker than $\Phi'(u_{m_j})=0$. Therefore, the fact that $(u_{m_j})$ converges strongly, up to a subsequence, to $u$ in $X$ does not follow from the usual standard argument.\\
Let us denote by $\Pi_{m_j}:X\to Y_{m_j}$ the orthogonal projection. Then it is clear that $\Pi_{m_j}u\to u$ in $X$, as $m_j\to\infty$. We have
\begin{multline}\label{eee}
\big<\Phi'_{m_j}(u_{m_j}),u_{{m_j}}-\Pi_{m_j}u\big>=\big(a+b\|u_{m_j}\|^2\big)\big<u_{m_j},u_{m_j}-\Pi_{m_j}u\big>\\
-\int_\Omega \big(u_{m_j}-\Pi_{m_j}u\big)f(x,u_{m_j})dx.
\end{multline}
Since $(u_{m_j})$ is bounded, we deduce from $(f_1)$ that $\big(|f(x,u_{m_j})|_{p/p-1}\big)$ is bounded. We then obtain by using the H\"{o}lder inequality
\begin{equation*}
\big|\int_\Omega \big(u_{m_j}-\Pi_{m_j}u\big)f(x,u_{m_j})dx\big|\leq|u_{m_j}-\Pi_{m_j}u|_p|f(x,u_{m_j})|_{\frac{p}{p-1}}\to0.
\end{equation*}
Recalling that $\Phi'_{m_j}(u_{m_j})=0$, we deduce from \eqref{eee} that
\begin{equation*}
\big<u_{m_j},u_{m_j}-\Pi_{m_j}u\big>=\|u_{m_j}\|^2-\big<u_{m_j},u\big>+\big<u_{m_j},u-\Pi_{m_j}u\big>=\circ(1).
\end{equation*}
It then follows that $\|u_{m_j}\|\to \|u\|$ which implies, since $X$ is uniformly convex, that $u_{m_j}\to u$ in $X$. It is readily seen that $u$ is a critical point of $\Phi$.

To show that the limit $u$ is sign-changing, we first observe that
\begin{align*}
\big<\Phi'_{m_j}(u_{m_j}),u_{m_j}^\pm\big>=0\quad&\Leftrightarrow\quad \big(a+b\|u_{m_j}\|^2\big)\|u_{m_j}^\pm\|^2=\int_\Omega u_{m_j}^\pm f(x,u_{m_j})dx\\
&\Rightarrow\quad a\|u_{m_j}^\pm\|^2\leq \int_\Omega u_{m_j}^\pm f(x,u_{m_j}^\pm)dx.
\end{align*}
By using \eqref{feps} and the Sobolev embedding theorem, we obtain
\begin{equation*}
a\|u_{m_j}^\pm\|^2\leq \int_\Omega u_{m_j}^\pm f(x,u_{m_j}^\pm)dx\leq c\big(\varepsilon\|u_{m_j}^\pm\|^2+c_\varepsilon\|u_{m_j}^\pm\|^p\big),
\end{equation*} 
where $c>0$ is a constant. Since $u_{m_j}$ is sign-changing, $u_{m_j}^\pm$ are not equal to $0$. Choosing $\varepsilon$ small enough \big(for instance $\varepsilon<\frac{a}{2c}$\big), we see that $(\|u_{m_j}^\pm\|)$ are bounded below by strictly positive constants which do not depend on $m_j$. This implies that the limit $u$ of the sequence $(u_{m_j})$ is also sign-changing.
\end{proof}
We are now in the position of proving our main result.
\begin{proof}[Proof Theorem \ref{mainresult}]
By Lemmas \ref{akbk}, \ref{Bop}, \ref{tech}, and \ref{psnod}, and Remarks \ref{a2} and \ref{rema}, the conditions $(A_1)$, $(A_2)$, $(A_3)$ and $(A_4)$ of Theorem \ref{scft} are satisfied. It then suffices to apply Theorem \ref{scft} to conclude.
\end{proof}



%
%
\end{document}